\documentclass[]{amsart}

\usepackage[utf8]{inputenc}
\usepackage{amssymb,amsmath,amsthm, enumitem, tikz-cd}
\usepackage[hidelinks]{hyperref}

\theoremstyle{plain}
\newtheorem{theorem}{Theorem}[section]
\newtheorem{lemma}[theorem]{Lemma}
\newtheorem*{lemma*}{Lemma}
\newtheorem{corollary}[theorem]{Corollary}
\newtheorem{proposition}[theorem]{Proposition}
\newtheorem*{proposition*}{Proposition}

\theoremstyle{definition}

\newtheorem{conjecture*}{Conjecture}

\newtheorem{question}[theorem]{Question}
\newtheorem*{question*}{Question}
\newtheorem{facts}[theorem]{Facts}
\newtheorem*{facts*}{Facts}
\newtheorem*{convention*}{Convention}
\newtheorem*{acknowledgements*}{Acknowledgements}

\theoremstyle{remark}
\newtheorem{remark}[theorem]{Remark}
\newtheorem*{remark*}{Remark}

\newcommand{\CC}{\mathbb{C}}
\newcommand{\GG}{\mathbb{G}}
\newcommand{\QQ}{\mathbb{Q}}
\newcommand{\ZZ}{\mathbb{Z}}
\newcommand{\Lcal}{\mathcal{L}}
\newcommand{\Ocal}{\mathcal{O}}
\renewcommand{\mod}{\operatorname{\mathsf{mod}}}
\newcommand{\Perf}{\operatorname{\mathsf{Perf}}}
\newcommand{\Coh}{\mathsf{Coh}}

\DeclareMathOperator{\End}{End}
\DeclareMathOperator{\RHom}{RHom}
\DeclareMathOperator{\Gr}{Gr}
\DeclareMathOperator{\Hom}{Hom}
\DeclareMathOperator{\proj}{proj}

\DeclareMathOperator{\rad}{rad}
\DeclareMathOperator{\rk}{rk}
\DeclareMathOperator{\SL}{SL}
\DeclareMathOperator{\Spec}{Spec}
\DeclareMathOperator{\rSpec}{\underline{Spec}}
\DeclareMathOperator{\Sym}{Sym}

\DeclareMathOperator{\Z}{Z}

\title[affine cones over Grassmannians and their stringy $E$-functions]{A note on affine cones over Grassmannians and their stringy $E$-functions}
\author{Timothy De Deyn}
\address{Departement Wiskunde, Vrije Universiteit
	Brussel, Pleinlaan $2$, B-1050 Elsene}
\email{timothy.de.deyn@vub.be}
\thanks{This work was supported by the Research Foundation - Flanders (FWO) Ph.D.\ Fellowship fundamental research 95249.}

\begin{document}
	\begin{abstract}
		We compute the stringy $E$-function of the affine cone over a Grassmannian. 
		If the Grassmannian is not a projective space then its cone does not admit a crepant resolution. Nonetheless the stringy $E$-function is sometimes a polynomial and in those cases the cone admits a noncommutative crepant resolution.
		This raises the question as to whether the existence of a noncommutative crepant resolution implies that the stringy $E$-function is a polynomial.
	\end{abstract}
	\maketitle

	\section{Introduction}
		In \cite{Batyrev1998} Batyrev defined the notion of a `stringy $E$-function', which is an invariant defined for any normal irreducible algebraic variety with at worst log-terminal singularities.
		An easy consequence of the definition is that a variety admitting a crepant resolution has a polynomial stringy $E$-function.
		One can wonder whether the same holds for `noncommutative analogues' of crepant resolutions.
		
		In this note we show the following\footnote{More generally we compute the stringy $E$-function for certain affine cones over Fano varieties, see Proposition \ref{prop1}.}.
		\begin{proposition*}[Propositions \ref{prop1}  and \ref{prop: Est poly iff gcd=1}]
			Let $X$ denote the affine cone over the Grassmannian $\Gr(k,n)$. 
			Then its stringy $E$-function
			\[
				E_{st}(X;u,v)	= \binom{n}{k}_{uv} \frac{(uv-1){uv}^n}{(uv)^n-1}.
			\]
			In particular, this is a polynomial if and only if $\gcd(k,n)=1$. 
		\end{proposition*}
		This gives an easier negative answer to a question asked by Batyrev \cite[Question 5.5]{Batyrev1998}, namely whether a GIT quotient of $\CC^n$ modulo an action of a semisimple subgroup $G\subset \SL_n(\CC)$ has a polynomial stringy $E$-function, than the one previously known \cite{Kiem}.
		
		Furthermore, this affine cone never admits a crepant resolution (if $k\neq 1$ or $n-1$).
		However, when $\gcd(k,n)=1$ it does admit a noncommutative crepant resolution (NCCR) (and a strongly crepant categorical resolution), which raises the question:
		\begin{question*}[Question \ref{ques: 1}]
			If a variety admits a NCCR (or a strongly crepant categorical resolution), is its stringy $E$-function a polynomial?
		\end{question*}
		Moreover, we show that the `stringy Euler characteristic' of the affine cone over the Grassmannian is equal to the rank of the Grothendieck group and to the Euler characteristic (computed via periodic cyclic homology) of its NCCR, see Proposition \ref{prop2}.
		The following question thus seems reasonable.
		\begin{question*}[Question \ref{ques: 2}]
			Does the stringy Euler characteristic of a variety equal the Euler characteristic (computed via periodic cyclic homology) of a NCCR?
		\end{question*}
		This indicates that NCCRs, if they exist, seem to encode interesting `stringy’ information of a variety.
		
		\begin{acknowledgements*}
			The author would like to thank Michel Van den Bergh for generously sharing his knowledge and for useful comments on a first draft of this note.
			Moreover, the author would like to thank Geoffrey Janssens, Theo Raedschelders and \v{S}pela \v{S}penko for useful conversations and helpful comments on a preliminary version of this note.
			Lastly, the author thanks the anonymous referee for carefully reading the note and for valuable suggestions.
		\end{acknowledgements*}
	
		\begin{convention*}
			Throughout we work over the complex numbers $\CC$.
		\end{convention*}

	\section{Preliminaries}
		\subsection{\texorpdfstring{$E$}{E}-function}
		Let $X$ be a variety of pure dimension $d$. 
		The cohomology groups with compact supports $H^i_c(X, \QQ)$ ($0\leq i\leq 2d$) carry a natural mixed Hodge structure.
		One defines the $E$-polynomial (also called the Hodge-Deligne polynomial) by
		\[
			E(X; u, v) = \sum_{i=0}^{2d} (-1)^i \sum_{p,q}	h^{p,q}(H^i_c(X,\CC))u^p v^q,
		\]
		where $h^{p,q}(H^i_c(X,\CC))$ denotes the dimension of the $(p, q)$-type Hodge component in $H^i_c(X, \CC)$.
		
		We recall some well-known facts that will be used below.
		\begin{facts}\label{facts}\hfill
			\begin{enumerate}
				\item If $Y\to X$ is a Zariski locally trivial fibration with fibre $F$, then $E(Y;u,v)=E(F;u,v)E(X;u,v)$, see e.g.\ \cite[Corollary 1.9]{DanilovKhovanski}.
				\item $E(\CC^*;u,v)=uv-1$.
                \item $E(\Gr(k,n);u,v)= \binom{n}{k}_{uv}$ is the Gaussian binomial coefficient (also called $q$-binomial coefficient), defined as
				\[
					\binom{n}{k}_q
					:= 
					\frac{(1-q^n)(1-q^{n-1})\cdots(1-q^{n-k+1})} {(1-q)(1-q^2)\cdots(1-q^k)}.
				\]
			\end{enumerate}
		\end{facts}
		
		\subsection{Stringy \texorpdfstring{$E$}{E}-function}\label{subsec: E_st}
		Now let $X$ be a normal irreducible variety with at worst log-terminal singularities.
		The stringy $E$-function of $X$ is defined as a certain motivic integral over its formal arc space, see e.g.\ \cite{CLNS, viusos}.
		However, to actually compute the integral, one usually performs a `change of variables' along a log resolution of $X$.
		We will only give an expression for the stringy $E$-polynomial after this change of variables.
		
		Thus, let $f:Y\to X$ be a log resolution, i.e.\ proper birational morphism with $Y$ smooth and such that the exceptional
		locus is a simple normal crossing (snc) divisor $D$. 
		Denote the irreducible components of $D$ by $D_1,\dots,D_r$ and set $I:=\{1,\dots r\}$.
		Define, for any subset $J\subseteq I$, 
		\[
			D_J:= \begin{cases}
				\bigcap_{j\in J}D_j&\text{if }J\neq\emptyset\\
				Y&\text{if }J=\emptyset
			\end{cases}
		\quad\text{and}\quad D^\circ_J:=D_J-\bigcup_{i\in I-J}D_i.
		\]
		Moreover, let $a_i$ denote the discrepancy  coefficient of $D_i$, i.e.\ $K_Y-f^*K_X=\sum_{i=1}^r a_i D_i$.
		
		After performing the change of variables one arrives at the following expression for the stringy $E$-function, which is how it was originally introduced in \cite{Batyrev1998},
		\[
			E_{st}(X;u,v)= \sum_{J\subseteq I} E( D^\circ_J;u,v)\prod_{j\in J} \frac{uv-1}{(uv)^{a_j+1}-1},
		\]
		where $\prod_{j\in J}$ is $1$ when $J=\emptyset$. 
	
		\subsection{Resolving cones}
		Let $V$ be an irreducible smooth projective variety with an ample line bundle $\Lcal$.
		Moreover, define the section ring $S:=\oplus_{k\geq 0} H^0(V,\Lcal^{\otimes k})$ and let $C:=\Spec S$ be the affine cone over $V$ with respect to $\Lcal$. 
		\begin{lemma}\label{lem: blowup cone}
			With notation as above. 
			The cone $C$ has a resolution given by the vector bundle\footnote{We use the convention that $\mathbf{V}(\Lcal):=\rSpec_V\left( \Sym (\Lcal)\right)$.} $\mathbf{V}(\Lcal)$ with exceptional locus given by the zero section.
		\end{lemma}
		\begin{proof}
			This is \cite[\href{https://stacks.math.columbia.edu/tag/0EKI}{Lemma 0EKI}]{stacks-project}.
		\end{proof}
		\begin{remark}
			If $S$ is generated by $S_1$ as an $S_0$-algebra, e.g.\ if $\Lcal$ is very ample and $V$ is projectively normal with respect to the closed embedding defined by $\Lcal$, then $\mathbf{V}(\Lcal)$ is isomorphic to the blow-up of $C$ in the irrelevant ideal $S_+$.
			Moreover, under this isomorphism the exceptional divisor of the blow-up is scheme-theoretically identified with the zero section of the vector bundle.
		\end{remark}
	
	\section{Stringy \texorpdfstring{$E$}{E}-functions for cones over smooth Fano varieties}
		
		\subsection{Set-up}
		Let $V$ be an irreducible smooth projective variety with canonical sheaf being some negative multiple of an ample line bundle $\Lcal$, i.e.\ $\omega_V=\Lcal^{\otimes-n}$ for some $n>0$.
		Our aim in this section is to calculate the stringy $E$-function of the affine cone $X$ over $V$ with respect to $\Lcal$. 
		\begin{proposition}\label{prop1}
			Let $V$ and $X$ be as above. Then $X$ is Gorenstein and
			\[
				E_{st}(X;u,v)= E(V;u,v) \frac{(uv-1)(uv)^n}{(uv)^n-1}.
			\]
		\end{proposition}
		\begin{proof}
		The fact that $X$ is Gorenstein follows by the assumptions on $V$ together with Kodaira vanishing and \cite[Lemma 7.1]{VdB04}.
		
		As explained in \S \ref{subsec: E_st}, in order to compute the stringy $E$-function we need a log resolution.
		This is obtained via Lemma \ref{lem: blowup cone}. 
		The cone $X$ is resolved by $Y:=\mathbf{V}(\Lcal)$ with exceptional locus $D$ the zero section (which is isomorphic to V). 
		As the exceptional locus is irreducible and smooth, this is indeed a log resolution.
		Therefore, in order to calculate the stringy $E$-function, it suffices to determine the discrepancy.
		This is done in \S 3.2 and then $E_{st}(X;u,v)$ is calculated in \S 3.3.
		\end{proof}

		\subsection{Determining the discrepancy}
		Let $f:Y\to X$ denote the resolution.
		Write
		\[
		\omega_Y=f^*\omega_X\otimes_{\Ocal_Y} \Ocal_Y(D)^{\otimes a}
		\]
		for some $a$.
		By restricting to $D$ and using respectively the adjunction formula, the fact that $\omega_V=\Lcal^{\otimes-n}$, the normal sheaf of the inclusion $D\subset Y$ is $\mathcal{N}_{D/Y}=\Ocal_Y(D)|_D=\Lcal^{\otimes-1}$ and\footnote{As $X$ is Gorenstein, $\omega_X$ is invertible so the claim follows from the commutative diagram
			\[
			\begin{tikzcd}[ampersand replacement=\&, row sep= scriptsize, column sep= scriptsize]
				D\arrow[r, hook]\arrow[d]	\& Y\arrow[d, "f"] \\
				\Spec\CC\arrow[r, hook, "p"] \& X\rlap{ ,}
			\end{tikzcd}
			\]
			where $p$ corresponds to the singular point.
		} 
		$(f^*\omega_X)|_D=\Ocal_D$, one obtains
		\[
		\Lcal^{\otimes-n+1}=\Lcal^{-a}.
		\]
		Hence, as $\Lcal$ is ample and therefore torsion-free in the Picard group, $a$ equals $n-1$.
		
		\subsection{Putting it together}
		Using Facts \ref{facts} we find
		\begin{align*}
			E_{st}(X;u,v)	&= E(Y-D;u,v)+E(D;u,v)\frac{uv-1}{(uv)^n-1} \\
			&= E(\CC^*;u,v)E(V;u,v)+E(V;q)\frac{uv-1}{(uv)^n-1} \\
			&= (uv-1)E(V;u,v)+E(V;u,v)\frac{uv-1}{(uv)^n-1} \\
			&= E(V;u,v) \frac{(uv-1)(uv)^n}{(uv)^n-1}.
		\end{align*}		
	\begin{remark}
		The same calculation holds in the setting of \cite[Example 5.1]{Batyrev1998}, i.e.\ when $\omega_V^{\otimes-l}=\Lcal^{\otimes k}$ for positive integers $k$ and $l$. 
		In this case however $X$ is no longer Gorenstein, but it is $\QQ$-Gorenstein.
		One obtains as discrepancy $k/l-1$ and
		\[
			E_{st}(X;u,v) = E(V;u,v) \frac{(uv-1)(uv)^{k/l}}{(uv)^{k/l}-1},
		\]
		as $\omega_X^{[l]}=\Ocal_X(lK_X)$ is invertible.
	\end{remark}
		
	\section{Cones over Grassmannians}
		\subsection{Set-up and stringy \texorpdfstring{$E$}{E}-function}
		Fix two integers\footnote{If $k=1$ or $n-1$ the cone is non-singular.} $1 < k < n - 1$ and let $\GG$ denote the Grassmannian $\Gr(k,n)$. 
		Then $\omega_\GG=\Ocal_\GG(-n)$ (for the Pl\"{u}cker embedding), and hence we can deduce the stringy $E$-function of $X$, the affine cone over $\GG$ with respect to $\Ocal_\GG(1)$, by Proposition \ref{prop1}.
		As the only dependency in the stringy $E$-function is in $uv$, we introduce $q:=uv$ and write the function in this variable for ease of notation.
		We have, using Facts \ref{facts},
		\begin{align*}
			E_{st}(X;q)	= \binom{n}{k}_q \frac{(q-1)q^n}{q^n-1}.
		\end{align*}
		We are interested in when this function is a polynomial.
		As the Gaussian binomial coefficients are always polynomials, this boils down to understanding their irreducible factors.
		\begin{proposition}\label{prop: Est poly iff gcd=1}
			The stringy $E$-function $E_{st}(X;q)$ is a polynomial if and only if $\gcd(k,n)=1$.
		\end{proposition}
		\begin{proof}
			Let $\Phi_j(q)$ denote the $j$th cyclotomic polynomial. 
			It is well-known that
			\[
				q^n-1 = \prod_{j\mid n}\Phi_j(q).
			\]
			The result now follows since for $j\mid n$, by \cite[Proposition 3.]{GaussianCoef}, $\Phi_j(q)$ divides $\binom{n}{k}_{q}$ if and only if $j\nmid k$.
		\end{proof}
		There is an alternative description of $X$ as an $\SL_k(\CC)$ quotient.
		Let $V$ and $W$ be respectively $k$- and $n$-dimensional vector spaces over $\CC$. 
		Consider the standard action of $\SL(V)$ on $\Hom(W,V)$, then $X=\Hom(W,V)/\!\!/\SL(V)$.
		Thus Proposition \ref{prop: Est poly iff gcd=1} immediately implies the following.
		\begin{corollary}
			When $\gcd(k,n)\neq1$ the affine cone over the Grassmannian $\Gr(k,n)$ gives a counterexample to \cite[Question 5.5]{Batyrev1998}, namely whether a GIT quotient of $\CC^n$ modulo an action of a semisimple subgroup $G\subset \SL_n(\CC)$ has a polynomial stringy $E$-function.
		\end{corollary}
		
		\subsection{(Noncommutative) crepant resolutions}
		As $X$ is an $\SL_k(\CC)$ quotient of a smooth variety and $\SL_k(\CC)$ is connected with no 
		nontrivial characters, $X$ is factorial \cite[Theorem 3.17.]{AGIV}.
		In addition, by the previous section we see that $X$ has at worst terminal singularities.
		Therefore $X$ cannot admit a crepant resolution (factoriality implies that the exceptional locus of any resolution has codimension 1 \cite{mathoverflow1}).
		
		However, if $\gcd(k,n)=1$ it was shown in \cite{SVdB} (for $k=2$) and \cite{Doyle} (for general $k$) that $X$ does admit a NCCR (see \S 4.3 for more details).
		Moreover, $X$ always admits a weakly crepant categorical resolution and it admits a strongly crepant categorical resolution when $\gcd(k,n)=1$, see \cite{Kuz} for $k=2$ whilst it follows from loc.\ cit.\ and the semi-orthogonal decomposition constructed in \cite{Fonarev} for general $k$.
        This together with Proposition \ref{prop: Est poly iff gcd=1} raises the following question.
        \begin{question}\label{ques: 1}
        	If a variety admits a NCCR (or a strongly crepant categorical resolution), is its stringy $E$-function a polynomial?
        \end{question}
    	At this point it is worthwhile to indicate the relationship between NCCRs and strongly crepant categorical resolutions. However, in order not to disturb the flow of the text, this is done in Appendix \ref{appendix}.
		
		\begin{remark}
			As far as we are aware there are no general results showing the non-existence of NCCRs when $\gcd(k,n)\neq 1$. 
			However, for $k=2$ and $n=4$ the non-existence can be proven using results of \cite{Dao}.
			Of course, if Question \ref{ques: 1} has an affirmative answer the non-existence would follow.
			In fact, one of the main motivations behind Question \ref{ques: 1} is to obtain a criterion to show the non-existence of NCCRs.
		\end{remark}    
    
		\begin{remark}
			One can easily check that every example in \cite[Section 7]{Kuz} admitting a strongly crepant categorical resolution has a polynomial stringy $E$-function.
		\end{remark}
	
		\begin{remark}
			As the case $\gcd(k,n)\neq1$ shows, admitting a weakly crepant categorical resolution is not enough to imply that the stringy $E$-function is a polynomial. 
		\end{remark}
		
		\begin{remark}
			The negative answer to \cite[Question 5.5]{Batyrev1998} in \cite{Kiem} comes from the GIT quotient of $\SL_2(\CC)$ acting diagonally on three copies of the adjoint representation.
			It is interesting to note that this singularity has a twisted NCCR by \cite[Theorem 1.13]{SVdB}. 
			Hence admitting a twisted NCCR is not enough to imply that the stringy $E$-function is a polynomial. 
		\end{remark}
	
		\begin{remark}
		It is known that Gorenstein toric varieties and Gorenstein quotient singularities have polynomial stringy $E$-functions \cite[Proposition 4.4]{Batyrev1998} and \cite{BatyrevDais}, \cite[Theorem 3.6]{DL2}. 
		This is compatible with affine Gorenstein quotient singularities (of the form $\CC^n/\!\!/ G$ with $G\subset\SL_n$ finite) always admitting a NCCR given by the twisted group ring, and that conjecturally affine Gorenstein toric varieties always admit a NCCR. 
		\end{remark}
	
		\subsection{Stringy Euler characteristic}
		The stringy Euler characteristic of $X$ is defined by
		\[
			e_{st}(X):=\lim_{q\to 1} E_{st}(X; q) = \binom{n}{k}\frac{1}{n}.
		\]
		When $\gcd(k,n)=1$ this is an integer and in this subsection we show that it is equal to the rank of the Grothendieck group and to the Euler characteristic (computed via periodic cyclic homology) of the NCCR of $X$.
		
		We briefly give the explicit form of the NCCR $\Lambda$ constructed in \cite{Doyle}.
		Recall that $X=X'/\!\!/ G$ where $X':=\Hom(W,V)$ with $V$ and $W$ respectively $k$- and $n$-dimensional vector spaces and $G:=\SL(V)$.
		Writing $\CC[X']$ for the coordinate ring of $X'$, we have
		\[
			\Lambda := \End_{ \CC[X']^{G}} \left(\bigoplus_{\alpha \in \mathcal{UP}_{n,k}} (\mathbb{S}^{\alpha}V^{*}\otimes_\CC \CC[X'])^{G} \right), 
		\]
		where $\mathcal{UP}_{n,k}$ denotes the Young diagrams that fit inside a triangle of length $n-k$ and height $k$. 
		Since $G$ acts via graded automorphisms on $\CC[X']$, the invariant ring, modules of covariants and the NCCR have a natural grading\footnote{
			The grading on the invariant ring  $\CC[X']^{G}$ is, up to taking a $k$-Veronese, the same grading as the one obtained from $X$ as a cone, see e.g. \cite[Proposition 16.1.]{RSVdB}.
		}\textsuperscript{,}\footnote{\label{fn: covariants}
			Define for any finite dimensional $G$-representation $U$ the module of covariants $M(U):=(U\otimes_\CC \CC[X'])^G$. 
			Let us write $\mathcal{U}=\oplus_{\alpha \in \mathcal{UP}_{n,k}} \mathbb{S}^{\alpha}V^{*}$ for ease of notation.
			As the $G$-action is generic we have $\Lambda=\End(M(\mathcal{U}))\cong M(\End(\mathcal{U}))$, this follows from \cite[Lemma 3.3]{SVdB}, and this isomorphism is compatible with the natural gradings on either side.
			The LHS has the usual grading on morphisms of graded modules over a graded ring.
			The RHS has a grading from viewing it as a graded subring of $\End(\mathcal{U})\otimes_\CC \CC[X']$.
		}. 
		
		Let $K_0(\Lambda):=K_0 \proj(\Lambda)$ denote the Grothendieck group of finitely generated projective $\Lambda$-modules and let $K_0^{gr}(\Lambda):=K_0 \mathop\mathrm{gr\text{-}proj}(\Lambda)$ denote the graded version.
		As the category of finitely generated projective graded $\Lambda$-modules is Krull-Schmidt\footnote{
			The graded algebra $\Lambda$ is finite dimensional in every degree.
			Therefore, the graded endomorphism ring of any finitely generated graded module is finite dimensional, and hence a local ring if the module is indecomposable (idempotents lift modulo the Jacobson radical).
		}, 
		$K_0^{gr}(\Lambda)$ has a $\ZZ[x,x^{-1}]$-basis given by the non-isomorphic indecomposable graded summands of $\Lambda$.
		Therefore, by \cite[Corollary 6.4.2]{Hazrat}, $K_0(\Lambda)$ has rank equal to the cardinality of this basis, which is the number of non-isomorphic summands of the reflexive module whose endomorphism ring is $\Lambda$.
		Thus, in our case the rank of $K_0(\Lambda)$ equals $|\mathcal{UP}_{n,k}|$. This is exactly $e_{st}(X)$.
		
		Alternatively, viewing $\Lambda$ as an algebra of covariants (see footnote \ref{fn: covariants}) one sees that the grading is positive and has grade zero piece equal to $|\mathcal{UP}_{n,k}|$ copies of $\CC$. 
		But by \cite[Theorem 6]{BassHellerSwan} $K_0(\Lambda)=K_0(\Lambda_0)$, so again we see $\mathop\mathrm{rk} K_0(\Lambda)=|\mathcal{UP}_{n,k}|=e_{st}(X)$.
		
		The second way of computing $K_0$ is applicable to the periodic cyclic homology of $\Lambda$.
        Using $\mathbb A^ 1$-invariance of periodic cyclic homology \cite[(3.13)]{Kassel} and \cite[Theorem 1.1]{Hoobler} we find
		\[
		HP_i(\Lambda)=HP_i(\Lambda_0)=		\begin{cases}
												\CC^{|\mathcal{UP}_{n,k}|} &\text{for }i=0,\\
												0 &\text{for }i=1.
											\end{cases}
		\]
		Thus the Euler characteristic of $\Lambda$,	$e(\Lambda):= \dim HP_0(\Lambda)-\dim HP_1(\Lambda)= e_{st}(X)$.
		In conclusion we have
		\begin{proposition}\label{prop2}
			Let $X$ be the affine cone over the Grassmanian $\Gr(k,n)$ with $\gcd(k,n)=1$ and let $\Lambda$ denote its NCCR. Then
			\[
				e_{st}(X)=\rk K_0(\Lambda) = e(\Lambda).
			\]
		\end{proposition}
		\begin{remark}
			We did not need to use that $\Lambda_0$ is equal to $|\mathcal{UP}_{n,k}|$ copies of $\CC$ to conclude  $e(\Lambda)=e_{st}(X)$. 
			The first way of showing $\mathop\mathrm{rk} K_0(\Lambda) =	e_{st}(X)$ did not use this, it was the fact that $\Lambda$ is an endomorphism ring which was relevant.
			Then merely by using that $\Lambda_0$ is finite dimensional we find 
			\[
			\rk K_0(\Lambda) = \rk K_0(\Lambda_0) = \rk K_0(\Lambda_0/\rad\Lambda_0) = e(\Lambda_0/\rad\Lambda_0)  = e(\Lambda_0) = e(\Lambda),
			\]
			where in the fourth equation we used \cite[Theorem II.5.1]{Goodwillie}. (Of course $\rad\Lambda_0=0$, but we did not need to know this.)
		\end{remark}
	
		This leads to the following question.
		\begin{question}\label{ques: 2}
			Does the stringy Euler characteristic of a variety equal the Euler characteristic (computed via periodic cyclic homology) of a NCCR?
		\end{question}

		\begin{remark}
			It is shown in \cite[Proposition 4.10]{Batyrev1998} that the stringy Euler characteristic of a normal $\QQ$-Gorenstein toric variety defined by a fan $\Sigma$ is equal to the volume of the shed associated to $\Sigma$.
			Now, for example, for the NCCRs of toric affine varieties constructed in \cite{SVdB4, SVdB5} the Euler characteristic also equals this volume by \cite[Theorem A.1]{SVdB4}, the derived invariance of $K_0$ and a similar reasoning as above equating the rank of $K_0$ and the Euler characteristic.
			In addition, for affine Gorenstein quotient singularities (of the form $\CC^n/\!\!/ G$ with $G\subset\SL_n$ finite) it follows from (the proof of) \cite[Theorem 8.4]{Batyrev1999} that the stringy Euler characteristic equals the number of conjugacy classes in $G$, and hence equals the number of irreducible representation of $G$.
			This in turn equals the number of non-isomorphic indecomposable summands of the reflexive module defining a NCCR of the quotient singularity, see e.g.\ \cite[\S 1.2]{SVdB}.
			As above this number equals the Euler characteristic of the NCCR.
		\end{remark}
	
		\begin{remark}
			In \cite{BorisovWang} Borisov and Wang define `Clifford-stringy Euler characteristics' as a stringy invariant of a Clifford noncommutative variety $(\mathbb{P}W,\mathcal{B}_0)$ (notation of loc.cit.).
			It is worthwhile to note that this invariant is exactly the Euler characteristic of $\mathcal{B}_0$ calculated via periodic cyclic homology.
			This follows from their main theorem \cite[Theorem 7.2]{BorisovWang}, the equivalence $D^b(Y_W ) \cong D^b(\mathbb{P}W,\mathcal{B}_0)$ (notation of loc.cit.) and the derived invariance of periodic cyclic homology.
		\end{remark}
	
		\begin{remark}
			The idea of using noncommutative algebras to resolve singular varieties and compute relevant data appeared early on in physics, see e.g.\ \cite{BerLeigh}. 
			From this point of view, it is natural to expect a relation between suitable `string theoretic invariants' and noncommutative geometry.
		\end{remark}
	
	\appendix
	\section{Noncommutative crepant resolutions vs.\ strongly crepant categorical resolutions}\label{appendix}
		In this appendix we briefly indicate the relationship between NCCRs and strongly crepant categorical resolutions (in the affine setting).
		This is probably well-known to experts, but as far as the author is aware this is not written down anywhere explicitly, however see \cite[Lemma 1.2]{SVdB} and \cite[Lemma 2.2.1.]{VdBICM}.
	
		For the remainder of this appendix let $R$ denote a finitely generated (as $\CC$-algebra) Gorenstein normal ring. 
		By a \emph{noncommutative resolution} of $R$ we mean an $R$-algebra which can be written as the endomorphism ring of a finite reflexive $R$-module that is of finite global dimension. 
		A \emph{noncommutative crepant resolution} is a noncommutative resolution that is moreover maximal Cohen-Macaulay as an $R$-module. 
		For the definition of a \emph{(strongly crepant) categorical resolution} we refer to \cite[Definitions 3.2.~ and 3.5.]{Kuz}.
		(But for  us a smooth triangulated category is an algebraic triangulated category that is smooth in the dg sense.)
		We consider all modules as left modules.
	
		\begin{proposition}\label{prop: app}
			Let $\Lambda=\End_R(M)$ be a noncommutative resolution of $R$.
			\begin{enumerate}[label = {\textnormal{(}$*$\textnormal{)}}]
				\centering
				\item\label{*}
				 Assume that $M$ has no self-extensions as $\Lambda$-module.
			\end{enumerate}
			Then $\pi^*:\Perf(R)\to \Perf(\Lambda) : N \mapsto M\otimes_R^L N$ gives a categorical resolution of singularities of $\Spec R$.
			Furthermore, this categorical resolution is strongly crepant if and only if $\Lambda$ is a noncommutative crepant resolution.
		\end{proposition}
		\begin{proof}
			The functor $\pi^*$ has as right adjoint $\pi_*:=\RHom_\Lambda(M,-)$. 
			In order to show that $(\Perf(\Lambda),\allowbreak\pi_*,\pi^*)$ is a categorical resolution of singularities we have to show that $\Perf(\Lambda)$ is smooth and that the natural transformation 
			\begin{equation}\label{eq: unit}
				1_{\Perf(R)}\to\pi_*\pi^*
			\end{equation}
			is an isomorphism.
			
			The smoothness of $\Perf(\Lambda)$ follows as $\Perf(\Lambda)=D^b(\mod\Lambda)$, since $\Lambda$ has finite global dimension, and the latter is smooth by \cite[Theorem 5.1.]{ElaginLuntsSchnurer}.
			(Alternatively one can show the smoothness of $\Lambda$ directly using \cite[Lemma 4.2.]{stafford2008}.)
			
			To see that \eqref{eq: unit} is an isomorphism it suffices to note, since $R$ generates $\Perf(R)$, that $R\to\pi_*\pi^*R$ is an isomorphism  by \ref{*} and as $\End_\Lambda(M)=\Z(\Lambda)=R$ (the last equality follows from \cite[\href{https://stacks.math.columbia.edu/tag/0AV9}{Lemma 0AV9}]{stacks-project} as the inclusion $R\hookrightarrow\Z(\Lambda)$ is an isomorphism in codimension one by the normality of $R$).
			
			Finally, to show that the two notions of crepancy correspond to each other it suffices to observe that both are equivalent to the relative Serre functor being the identity.
			For strongly crepant categorical resolutions this is by definition, whilst for NCCRs this follows from (the proof of) \cite[Theorem 4.14.]{IW1}.
		\end{proof}
		
		\begin{remark}
			It is often the case that the reflexive module defining the noncommutative (crepant) resolution contains $R$ as a direct summand or is maximal Cohen-Macaulay.
			In both cases $M$ is a projective $\Lambda$-module, and hence satisfies \ref{*}.
			Moreover, \ref{*} is automatically satisfied when $\dim R\leq 3$ \cite[Proposition 2.15.]{VdBICM}.
		\end{remark}

		In general there is no reason to expect every categorical resolution $\mathcal{D}$ of $\Spec R$ to come from a noncommutative resolution.
		One compelling reason for this is that there is no `birationality' requirement in the definition of a categorical resolution of singularities (e.g.\ $D^b(\Coh(\mathbb{P}^n))$ is a categorical resolution of $D^b(\Coh(\Spec k))$).
		
		It is not really clear what the right categorical notion of `birationality' should be. 
		A sufficient condition that would enforce $\mathcal{D}$ to be equivalent to the derived category of a noncommutative resolution of $R$ would be the following.
		Suppose $\mathcal{D}$ has a tilting object with endomorphism ring $\Lambda$ being a finite reflexive $R$-module which is Morita equivalent to $R$ in codimension one. (This should be viewed as being analogous to `isomorphic in codimension one' for varieties.) 
		Then $\Lambda$ can be written as the endomorphism ring of a reflexive $R$-module \cite[Proposition 2.11]{IW}.
		Of course, as $\mathcal{D}$ is smooth $\Lambda$ has finite global dimension and hence is a noncommutative resolution of $R$.
		
		If the equivalence between $\mathcal{D}$ and $\Perf(\Lambda)$ is compatible with the $\Perf(R)$-module structures\footnote{In this context it seems natural to require $\mathcal{D}$ having an $R$-linear dg enhancement. Then, as the equivalence is obtained from an $R$-linear dg functor, the compatibility is automatic.}, then, as in the proof of Proposition \ref{prop: app}, $\mathcal{D}$ is strongly crepant if and only if $\Lambda$ is an NCCR.
		

	\bibliographystyle{amsalpha}
	\bibliography{nccr}
\end{document}